\documentclass[preprint,10pt]{article}
\usepackage[utf8]{inputenc}
\usepackage{latexsym}
\usepackage{amsfonts}

\usepackage{amssymb}
\usepackage{amsmath}
\usepackage{amsthm}

\usepackage{color}
\definecolor{mygray}{gray}{0.6}
\usepackage[colorlinks=true,citecolor=red,linkcolor=blue,pdfpagetransition=Blinds]{hyperref}

\usepackage[english]{babel}

\newtheorem{teorema}{Theorem}[section]

\newtheorem{lema}[teorema]{Lemma}
\newtheorem{defi}[teorema]{Definition}
\newtheorem{propo}[teorema]{Proposition}

\newtheorem{obs}[teorema]{Remark}

\newcommand\D{\displaystyle}
\allowdisplaybreaks

\usepackage[top=1in, bottom=1.25in, left=1.25in, right=1.25in]{geometry}
\providecommand{\keywords}[1]{\noindent {\textbf{Keywords:}} #1}
\providecommand{\AMScodes}[1]{\noindent {\textbf{MSC2010:}} #1}

\begin{document}

\title{{\bf Distribution function of the blow up time of the solution of an anticipating random fatigue equation}}

\author{Liliana Peralta \thanks{Centro de Investigaci\'on en Matem\'aticas, UAEH, Carretera Pachuca-Tulancingo km 4.5 Pachuca, Hidalgo 42184, Mexico  E-mail: \texttt{liliana\_peralta@uaeh.edu.mx}}
%Centro de Investigaci?on en Matem?aticas, UAEH,
%Carretera Pachuca-Tulancingo km 4.5 Pachuca,
%Hidalgo 42184, Mexico
%liliana peralta@uaeh.edu.mx
%\and Author 2\thanks{aqui las del segundo. E-mail: \texttt{perenganito@perez.com}}
}

\maketitle

\abstract{In this paper, we study the distribution function of the time of explosion of a stochastic differential equation modeling the length of the dominant crack due to fatigue. The main novelty is that initial condition is regarded as an anticipating random variable and the stochastic integral is in the forward sense.

Under suitable conditions, we use the substitution formula from Russo and Vallois to find the local solution of this equation. Then, we find the law of blow up time by proving some results on barrier crossing probabilities of Brownian bridge.}

\

\keywords{Stochastic differential equations, random fatigue, explosion time, forward integral, Brownian bridge.}

\AMScodes{60H10, 91B70, 60H05, 60J65}

\section{Introduction}
Nowadays, most of the components of certain types of structures present discontinuities or notches, which usually appear by the manufacturing process and operating conditions. With time and due to cyclic loads, these discontinuities cause the appearance of cracks, and in turn generate high concentration of efforts that could produce failures in the structures, whence, the rate of propagation of a crack varies in time. 

Over time, different models have been developed to predict the propagation of cracks, relating the properties of the material with geometric characteristics and different loading conditions. In 1963  the well known Paris' law \cite{paris} of fracture mechanics is introduced and, since then, several generalizations of this model have been made. As experiments have shown, the randomness is a characteristic feature of crack growth and stochastic components have been incorporated  to the models to take into account the typical variability of the dynamics of cracks on solids. 

The following stochastic model%
\begin{equation}
\label{PE}
dL_t=\left(c_1L_s^{p}+\frac{p c_2^2}{2}L_s^{2p-1}\right)dt+\int_0^t c_2L_s^{p}dW_t,\qquad L_0=l_0,
\end{equation}
is a generalization of the classical Paris' law and it was proposed by Sobczyk (\cite{sobc2}, \cite{rfsobc}) to represent the time evolution of the length of the dominant crack.  Here $W$ is a Brownian motion defined in a filtered probability space $(\Omega, \mathcal{F},\{\mathcal{F}_t\}_{t\geq 0}, P)$ which filtration satisfies the usual conditions, while $c_1$ and $c_2$ are positive constants depending on various parameters used to describe the effects of random loading on crack growth, in particular the intensity factor range and ratio of the applied stress. The variable $p$ is a positive constant and it is determined from experimental data (see, for instance \cite{frost}, \cite{valluri}). When $p=1$, the coefficients of equation (\ref{PE}) have linear growth and thus its solution is global, i.e., it is well defined for all $t\geq 0$. Nevertheless, it has been found that several materials have values of $p$ greater than one (see \cite{molenta}), varying according to the type of material, the environmental conditions and the type of load applied. In this case, the solution $L$ may explode in finite time (see \cite{arnold}, \cite{mck}) which, at the mechanical level, translates into the precise instant where the fracture occurs.   

The point of the fracture is a relevant topic and has attracted attention in the recent past due to its numerous applications. For instance in \cite{ford}, the authors investigate the mechanisms of fatigue crack growth in polycarbonate polyurethane, a material widely used in orthopedic applications due to its cartilage-like, hygroscopic, and elastomeric properties. Other interesting applications can be found in the fields of mechanics, aviation, shipping. We refer the interested reader to the works \cite{augustin}, \cite{chamis}, \cite{lin}.

In the present paper, we propose to find the distribution of the blow up time of the solution of process \eqref{PE} when $p>1$ but, different from \cite{sobc2} and other related works, we replace the deterministic initial condition $l_0$ by a more general anticipating random variable $I$. As a consequence, it will be necessary to use tools of anticipating stochastic calculus (see, for instance \cite{Arturo}, \cite{nualart}, \cite{pardoux}) in order to achieve the goal of the work. The anticipating stochastic calculus allows the study of stochastic differential equations (SDE) where the coefficients or the initial condition may depend on future information. 
%In the case of random fatigue this means that the length of the dominant crack at a certain time $T>0$ is already known when the model is seen at time $t=0$. 

The stochastic integral in equation \eqref{PE} is understood in the classical It\^o sense (see, e.g. \cite{ito}) and is defined for adapted processes to the information generated by the Brownian motion $W$, but in this proposal, taking into account the new characteristic of the initial condition, we need to use an integral which allows the integration of non-adapted processes. The anticipating stochastic integral known as forward, defined by Russo and Vallois in \cite{RV}, is an extension of the well known It\^o's integral for the case of anticipating integrands and coincides with the It\^o's one when the integrands are square integrable, measurable and $\mathcal{F}_t-$adapted processes.

To achieve the objective of this paper, first in Section \ref{AC}, we will be prove using the substitution formula for the forward integral \cite{RV}, that the local solution of the anticipating stochastic equation is given as follows:

\begin{teorema}
\label{solteo}
Let $\bar{t}>0$ and $I>0$ a random variable. Then the process 
\begin{equation*}
\left(\left(I\right)^{1-p}\!\!\!-c_1(p-1)_{.}-c_2(p-1)W_{.}\right)^{\frac{p}{1-p}}
\end{equation*}
belongs to $\text{Dom } \delta_{\bar{t}, loc}^{-}$ (see Definition \ref{local}) on the set
\begin{equation*}
A_{\bar{t}}=\left\{ \left(I\right)^{1-p}\!\!\!-c_1(p-1)\bar{t}-c_2(p-1)\sup_{0\leq s\leq \bar{t}}W_s>0\right\}
\end{equation*}
for any $t\in(0,\bar{t}\:]$. Furthermore, for $t\leq \bar{t}$, $p>1$, and $I>0$ we have that
\begin{equation}
\label{SE}
L^{I}_t=I+\int_0^t\left(c_1\left(L^{I}_s\right)^p+\frac{pc_2^2}{2}\left(L_s^{I}\right)^{2p-1}\right)ds+\int_0^tc_2\left(L_s^{I}\right)^pd^-W_s
\end{equation}
%f
is well defined on $A_{\bar{t}}$, where 
\begin{equation}
\label{solparis}
L^x_t=\left(\left(\frac{1}{x}\right)^{p-1}-c_1(p-1)t-c_2(p-1)W_t\right)^{\frac{1}{1-p}},\hspace{.4cm}t\geq 0, \:x\in\mathbb{R}.
\end{equation}
\end{teorema}

Since we are considering the case $p>1$, notice that \eqref{solparis} might have paths that reach infinite values in finite times. The first instant in which this phenomenon occurs is called the time of the explosion of the process and hereinafter we will denote it by $\tau$.

In Section \ref{LD}, when the initial condition for \eqref{SE} is given by $I=g(W_T)$ for a borel function $g:\mathbb{R}\to (0,\infty)$, we will compute the probability distribution of the time of explosion of the solution of process \eqref{SE}. To find this distribution, we will use results on barrier crossing probabilities of Brownian bridge (see e. g. \cite{anna},\cite{revuz}), namely, we will prove conditional crossing-probability for a linear boundary of the form
\begin{equation}
\label{browbridg}
\D P\left(\bigcup_{0\leq t \leq r}\left\{W_t\geq a-bt\right\}| W_T=x\right).
\end{equation}
The barrier crossing probabilities that we will prove are a generalization of those found in \cite{beg} and they have been extended for the case of piecewise-linear boundary in \cite{abundo}. For the sake of conciseness, we will present proofs that are useful and adapted for our purposes. Since the moment of the fracture may occur before or after the time $T$ we pay attention to the cases $r\leq T$ and $r> T$ in the probability \eqref{browbridg}.

With all of the above, we find that the distribution function of the explosion time of the solution of \eqref{SE} is:

\begin{teorema}
\label{FDdetau}
Let $\Phi$ be the standard normal distribution function, $\phi(0,\sigma^2)$ be the normal density function with mean zero and variance $\sigma^2$ and $\tau$ be the blow up time of solution of the anticipating process \eqref{SE} given by
\begin{equation*}
\tau=\inf\left\{t>0: W_t\geq \frac{1}{c_2(p-1)I^{p-1}}-\frac{c_1}{c_2}t\right\}.
\end{equation*}
Define 
\begin{align}\notag
a(x)&=\frac{1}{c_2(p-1)\left(g(x)\right)^{p-1}}-x-\frac{c_1T}{c_2}\text{   and   }v=\frac{1}{r}-\frac{1}{T}.
\end{align}
Then
\begin{description}
\item[i)] If $r<T$,  
\begin{align}\notag
P&(\tau\leq r)\\\notag
&=1-\int_{-\infty}^{\infty}\Phi\left(\frac{a(x)}{T\sqrt{v}}+\frac{\sqrt{v}}{c_2(p-1)\left(g(x)\right)^{p-1}}\right)\phi(0,T)dx\\\notag
&+\int_{-\infty}^{\infty}\exp\left(\frac{-2a(x)}{Tc_2(p-1)\left(g(x)\right)^{p-1}}\right)\Phi\left(\frac{a(x)}{T\sqrt{v}}-\frac{\sqrt{v}}{c_2(p-1)\left(g(x)\right)^{p-1}}\right)\phi(0,T)dx.
\end{align}
\item[ii)] If $r=T$,
\begin{align}\notag
P(\tau\leq r)=&1-\int_{-\infty}^{\infty}I_{\{a(x)>0\}}\phi(0,T)dx\\\notag
&+\int_{-\infty}^{\infty}\exp\left(\frac{-2a(x)}{Tc_2(p-1)\left(g(x)\right)^{p-1}}\right)I_{\{a(x)>0\}}\phi(0,T)dx.\\\notag
\end{align}
\item[iii)] If $r>T$,
\begin{equation*}
\begin{split}
&P(\tau\leq r)\\
&=\int_{-\infty}^{\infty}\exp\left(\frac{-2a(x)}{Tc_2(p-1)\left(g(x)\right)^{p-1}}\right)\Phi\left(\frac{a(x)}{\sqrt{r-T}}-\frac{c_1\sqrt{r-T}}{c_2}\right)I_{\{a(x)>0\}}\phi(0,T)dx\\
&\quad-\int_{-\infty}^{\infty}\exp\left(\frac{-2a(x)}{Tc_2(p-1)\left(g(x)\right)^{p-1}}+\frac{2a(x)c_1}{c_2}\right)\\
&\hspace{4.7cm}\times\Phi\left(\frac{-a(x)}{\sqrt{r-T}}-\frac{c_1\sqrt{r-T}}{c_2}\right)I_{\{a(x)>0\}}\phi(0,T)dx\\
&\quad+1-\int_{-\infty}^{\infty}\Phi\left(\frac{a(x)}{\sqrt{r-T}}-\frac{c_1\sqrt{r-T}}{c_2}\right)I_{\{a(x)>0\}}\phi(0,T)dx\\
&\quad+\int_{-\infty}^{\infty}\exp\left(\frac{2a(x)c_1}{c_2}\right)\Phi\left(\frac{-a(x)}{\sqrt{r-T}}-\frac{c_1\sqrt{r-T}}{c_2}\right)I_{\{a(x)>0\}}\phi(0,T)dx.
\end{split}
\end{equation*}
\end{description}
\end{teorema}
%%%%%%%%%%%%%%%%%%%%%%%%%%%%%%%%%%%%%%%%%%%%%%%%%%
%%%%%%%%%%%%%%%%%%%% p r e l i m i n a r e s %%%%%%%%%%%%%%%%%%%%%%%%%
%%%%%%%%%%%%%%%%%%%%%%%%%%%%%%%%%%%%%%%%%%%%%%%%%%%%
\section{The anticipating case}
\label{AC}
In this section we find the solution of equation \eqref{SE} in which the involved stochastic process
\begin{equation*}
\label{FI}
\left\{\int_0^tc_2\left(L_s^{I}\right)^pd^-W_s,\right\}_{t\geq0}
\end{equation*}
is a forward integral. The forward integral was introduced by Russo and Vallois in their seminal work \cite{RV} and is an extension of It\^o's one. We present its definition.

\begin{defi}
Let $T>0$ and $v$ be a measurable process with integrable trajectories. We say that $v$ is forward integrable (i.e. $v\in  \text{Dom } \delta_{T}^{-}$) if
\begin{equation*}
\frac{1}{\epsilon}\int_0^{T}v_s(W_{(s+\epsilon)\wedge T}-W_s)ds,
\end{equation*}
converges in probability as $\epsilon \downarrow 0$. We denote this limit by $\int_0^Tv_sd_s^{-}W_s$.
\end{defi}

In order to achieve our goal, we will use the so-called substitution formula for the forward integral (see, for instance, \cite{navarro}, \cite{RV}), which we enunciate below.   

To this end let $\mathcal{R}$ be the set of random fields $X=\{X_t(u): t\in[0,T], u\in\mathbb{R}\}$ which are $\mathcal{P}\otimes \mathcal{B}(\mathbb{R})-$measurables, where $\mathcal{P}$ is the $\sigma-$algebra generated by the previsible processes. We have the following: 

\begin{teorema}\label{forsubs}
Consider the class of processes 
\begin{equation*}
\begin{split}
\mathcal{R}_2=&\{ X \in \mathcal{R} :  X(0) \in L^2([0,T]),\hspace{.2cm}X_{t}(u)\hspace{.2cm}\hbox{is differentiable in}\\
&u\hspace{.2cm}and\hspace{.2cm}\int_{-n}^{n}\int_{0}^{T}X^{\prime}_{t}(u)^2dtdu<\infty \hspace{.2cm}\forall n\in \mathbb{N}\}.
\end{split}
\end{equation*}
If $X\in \mathcal{R}_2$, then for every random variable $Z$, $X(Z)\in \text{Dom } \delta_{T}^{-}$  and 
\begin{equation}
\label{subs}
\int_{0}^{T}X_{s}(Z)d^{-}W_s=\left( \int_{0}^{T}X_{s}(u)dW_s\right)_{u=Z}.
\end{equation}
\end{teorema}

Bearing in mind the above statements we shall look for the solution of process \eqref{PE} by using the substitution formula. First, to fix ideas we set $L_0=x$, for $x\in\mathbb{R}$. Transforming the process \eqref{PE} to a more convenient one, i.e., $\{(L_t)^{1-p}\}_{t\geq 0}$ and applying It\^o's formula, we find that the solution is given by
\begin{equation}
\label{sol}
L^x_t=\left(\left(\frac{1}{x}\right)^{p-1}-c_1(p-1)t-c_2(p-1)W_t\right)^{\frac{1}{1-p}},\hspace{.4cm}t\geq 0.
\end{equation}

Since $p>1$, the process \eqref{sol} may explode with positive probability (see \cite{mck}) and therefore it is necessary to use that the forward integral satisfies the following local property which was proved in \cite{navarro}. 

\begin{lema}
\label{lemaF}
Let $v, u \in \text{Dom } \delta_{T}^{-}$ be mesurable processes and $A\in \mathcal{F}$ such that
\begin{equation*}
u_t=v_t \quad \text{a.s. on}\quad A\times [0,T].
\end{equation*}
%x
Then $\int_0^Tv_sd_s^{-}W_s=\int_0^Tu_sd_s^{-}W_s$ a.s. on $A$.
\end{lema}

As a consequence of Lemma \ref{lemaF}, the domain of forward integral can be extended in the following manner. 

\begin{defi}
\label{local}
A process $v$ is locally forward integrable (i.e. $v\in  \text{Dom } \delta_{T, loc}^{-}$) on a measurable set $A\in \mathcal{F}$, if there exist a subsequence $\{(v^{n}, A_n)\}_{n\in \mathbb{N}}$ in $\text{Dom}\:\delta_{T}^{-}\times \mathcal{F}$, such that:
 
\begin{enumerate}
  \item[i)] $A_n \nearrow A$ a.s. 
  \item[ii)] $v=v^{n}$ a.s.  on $A_n\times [0,T]$.
\end{enumerate}

Thus, 
\begin{equation*}
\int_0^Tv_sd^{-}W_s\equiv \int_0^T v_s^n d^{-}W_s \quad \text{on}\quad A_n.
\end{equation*}
\end{defi}

\begin{obs}
Lemma \ref{lemaF} implies that Definition \ref{local} is independent of the localizing sequence $\{(v^{n}, A_n)\}_{n\in \mathbb{N}}$.
\end{obs}  

Before we prove the main result of this section, we shall need to define the following functions. For $m\in\mathbb{N}$ large enough, let $\varphi_m\in C^{\infty}(\mathbb{R})$ be bounded functions satisfying 
\begin{equation}
\label{fis}
\varphi_{m}(x)=
\begin{cases}
\frac{2}{m^{(p-1)/p}}, \quad x\leq \frac{2}{m^{(p-1)/p}} \\
\quad x, \hspace{1.02cm}  x\in(\frac{3}{m^{(p-1)/p}},m) \\
m+1, \hspace{0.68cm} x\geq m+1.
\end{cases}
\end{equation}
The existence of these functions can be verified on \cite{lee}. It is clear that $\varphi_m(x)>0$ for all $x\in\mathbb{R}$. Hence, for $t>0$, the It\^o's stochastic integrals
\begin{equation*}
\int_0^t c_2\left(\varphi_m\left(\left(\frac{1}{x}\right)^{p-1}\!\!\!-c_1(p-1)s-c_2(p-1)W_s\right)\right)^{\frac{p}{1-p}}dW_s,
\end{equation*}
are well-defined for any $t>0$. In addition, it is easy to see that the random fields 
\begin{equation*}
\left(\varphi_m\left(u^{p-1}-c_1(p-1)s-c_2(p-1)W_s\right)\right)^{\frac{p}{1-p}},
\end{equation*}
satisfy the hypotheses of Theorem \ref{forsubs} and therefore the forward integrals 
\begin{equation*}
\int_0^t c_2\left(\varphi_m\left(\left(I\right)^{1-p}\!\!\!-c_1(p-1)s-c_2(p-1)W_s\right)\right)^{\frac{p}{1-p}}d^-W_s,
\end{equation*}
are well-defined too. 

Now we are in position to present the main result of this section. 

\begin{proof}[Proof of the Theorem \ref{solteo}]
For $m\in\mathbb{N}$ large enough, we consider the sets
\begin{equation*}
\begin{split}
A_m=&\left\{ \left(I\right)^{1-p}\!\!\!-c_1(1-p)s-c_2(p-1)\sup_{0\leq s\leq \bar{t}}W_s\geq \frac{3}{m^{(p-1)/p}},\right.\\
 &\quad\left.\left(I\right)^{1-p}\!\!\!-m\leq c_2(p-1)\inf_{0\leq s \leq \bar{t}}W_s\right\}.
\end{split}
\end{equation*} 
Accordingly with the definition of functions \eqref{fis}, we can verify that
\begin{equation}
\label{locll}
\varphi_{m}\left(\left(I\right)^{1-p}\!\!\!-c_1(p-1)s-c_2(p-1)W_s\right)=\left(I\right)^{1-p}\!\!\!-c_1(p-1)s-c_2(p-1)W_s
\end{equation}
on $A_m\times [0,t]$ a.s., for $t\leq \bar{t}$. Thus, from the substitution formula \eqref{subs} in Theorem \ref{forsubs}  we get

\begin{align}\notag
1_{A_m}&\int_0^{t}c_2 \left(\varphi_{m}\left(\left(I\right)^{1-p}\!\!\!-c_1(p-1)s-c_2(p-1)W_s\right)\right)^{\frac{p}{1-p}}d^-W_s\\ \notag
&=\left(1_{\left\{x^{p-1}-c_1(p-1)s-c_2(p-1)\sup_{0\leq s\leq \bar{t}}W_s\geq \frac{3}{m^{(p-1)/p}}, x^{p-1}-m\leq c_2(p-1)\inf_{0\leq s \leq \bar{t}}W_s,\frac{1}{x}>0\right\}}\right.\\ \notag
&\left.\hspace{2.25cm}\times \int_0^t c_2 \left(\varphi_{m}\left(x^{p-1}-c_1(p-1)s-c_2(p-1)W_s\right)\right)^{\frac{p}{1-p}}dW_s\right)_{x=1/I}
\end{align}
and from equality \eqref{locll} we have
\begin{align}\notag
&\qquad=\left(1_{\left\{x^{p-1}-c_1(p-1)s-c_2(p-1)\sup_{0\leq s\leq \bar{t}}W_s\geq \frac{3}{m^{(p-1)/p}}, x^{p-1}-m\leq c_2(p-1)\inf_{0\leq s \leq \bar{t}}W_s,\frac{1}{x}>0\right\}}\right.\\ \notag
&\left.\hspace{2.88cm}\times \int_0^t c_2 \left(x^{p-1}-c_1(p-1)s-c_2(p-1)W_s\right)^{\frac{p}{1-p}}dW_s\right)_{x=1/I}.
\end{align}
Replacing the integrand in the above equation by the formula defined in \eqref{sol}, it is obtained that
\begin{align}\notag
1_{A_m}&\int_0^{t}c_2 \left(\varphi_{m}\left(\left(I\right)^{1-p}\!\!\!-c_1(p-1)s-c_2(p-1)W_s\right)\right)^{\frac{p}{1-p}}d^-W_s\\ \notag
&=\left(1_{\left\{x^{p-1}-c_1(p-1)s-c_2(p-1)\sup_{0\leq s\leq \bar{t}}W_s\geq \frac{3}{m^{(p-1)/p}}, x^{p-1}-m\leq c_2(p-1)\inf_{0\leq s \leq \bar{t}}W_s,\frac{1}{x}>0\right\}}\right.\\ \notag
&\left.\hspace{2.25cm}\times \int_0^t c_2\left(L_s^{\frac{1}{x}}\right)^{p}dW_s\right)_{x=1/I}\\\notag
&=\left(1_{\left\{x^{p-1}-c_1(p-1)s-c_2(p-1)\sup_{0\leq s\leq \bar{t}}W_s\geq \frac{3}{m^{(p-1)/p}}, x^{p-1}-m\leq c_2(p-1)\inf_{0\leq s \leq \bar{t}}W_s,\frac{1}{x}>0\right\}}\right.\\ \notag
&\left.\hspace{2.25cm}\times \left\{L_t^{\frac{1}{x}}-\frac{1}{x}-\int_0^t\left(c_1\left(L_s^{\frac{1}{x}}\right)^p+\frac{pc_2^2}{2}\left(L_s^{\frac{1}{x}}\right)^{2p-1}\right)ds\right\}\right)_{x=1/I}\\ \notag
&=1_{A_m}\left(L_t^{I}-I-\int_0^t \left(c_1\left(L_s^{I}\right)^p+\frac{pc_2^2}{2}\left(L_s^{I}\right)^{2p-1}\right)ds\right),
\end{align}
where the last equality follows from Fubini's Theorem. 

Finally, since $A_m \nearrow A_{\bar{t}}$ a.s., then $\{\varphi_{m}\left(\left(\frac{1}{I}\right)^{p-1}\!\!\!-c_1(p-1)s-c_2(p-1)W_s\right),A_m\}$ is a localizing sequence for the process $\left(\frac{1}{I}\right)^{p-1}\!\!\!-c_1(p-1).-c_2(p-1)W_{.}$ on $A_{\bar{t}}$. This concludes the proof.
\end{proof}

%%%%%%%%%%%%%%%%%%%%%%%%%%%%%%%%%%%%%%%%%%%%%%%%%%%%%%%%%%
%%%%%%%%%%%OTHER SECTION %%%%%%%%%%%%%%%%%%%%%%%%%%%%%%%%%%%%%%%%
%%%%%%%%%%%%%%%%%%%%%%%%%%%%%%%%%%%%%%%%%%%%%%%%%%%%%%%%%%
\section{Life Distribution}
\label{LD}
We have proved in Theorem \ref{solteo} that the solution of the process \eqref{SE} is given by
\begin{equation*}
L_t^{I}=\left(\left(I\right)^{1-p}-c_1(p-1)t-c_2(p-1)W_t\right)^{\frac{1}{1-p}},\hspace{.4cm}t\geq 0
\end{equation*}
and, as discussed before, we consider the random variable 
\begin{equation}
\label{bup}
\tau=\inf\left\{t>0: W_t\geq \frac{1}{c_2(p-1)I^{p-1}}-\frac{c_1}{c_2}t\right\}
\end{equation}
as its blow up time. 

To compute the distribution function of \eqref{bup}, we begin this section by proving crossing results for the Brownian bridge under conditions that will be sufficient for our main purpose. The proofs of this results follow a combination of invariance properties of standard Brownian motion in a similar fashion as \cite{scheike}.

We begin by recalling the following results (see, for instance \cite{borodin}, \cite{doob}, \cite{scheike}).

\begin{propo}
\label{clasico}
Let $W$ be an standard Brownian motion and let $g(t)=a+bt$. 
\begin{description}
  \item[i)] If $a,b>0$, then
\begin{equation*}
P\left(\bigcup_{0\leq t }\{W_t\geq g(t)\}\right)=  \exp(-2ab).
\end{equation*}

If either $a\leq 0$ or $b\leq 0$, then the probability is $1$. 
\item[ii)] If $a>0$ and $r<\infty$, then  
\begin{equation}
\label{err}
P\left(\bigcup_{0\leq t \leq r }\{W_t\geq g(t)\}\right)=1-\Phi\left(\frac{a}{\sqrt{r}}+b\sqrt{r}\right) +\exp(-2ab)\Phi\left(b\sqrt{r}-\frac{a}{\sqrt{r}}\right).
\end{equation}

If $a\leq 0$, then the probability is $1$. \footnote{In the original paper \cite{scheike} there is a mismatched sign in the last term of \eqref{err}. This was a typo that was later corrected in \cite{abundo}.}
\end{description}
\end{propo}

Now, we are ready to prove some crossing results concerning the Brownian bridge. The first result is the following. 

\begin{teorema}
\label{S/A}
Let $\Phi$ be the normal standard distribution, $a, b, x\in\mathbb{R}$, $T, r>0$ and define $v=\frac{1}{r}-\frac{1}{T}$.
\begin{description}
  \item[a)] If $a>0$ and $a-bT>x$, then 
  \begin{equation}
  \label{bb}
\D P\left(\bigcup_{0\leq t \leq T}\{W_t\geq a-bt\}| W_T=x\right)=\exp\left(-2a(a-x-bT)/T\right).
\end{equation}
If either $a\leq 0$ or $a-bT\leq x$, then the probability \eqref{bb} is 1.
\item[b)] Let $r<T$. If $a>0$, then
\begin{equation}
\label{bbcp2c}
\begin{split}
\D P&\left(\bigcup_{0\leq t \leq r}\{W_t\geq a-bt\}| W_T=x\right)\\
&\qquad =1-\Phi\left(\frac{a-x-bT}{T\sqrt{v}}+a\sqrt{v}\right)\\
&\qquad \quad+\exp\left(-2a(a-x-bT)/T\right)\Phi\left(\frac{a-x-bT}{T\sqrt{v}}-a\sqrt{v}\right).
\end{split}
\end{equation}

If $a\leq 0$, then the probability in \eqref{bbcp2c} is 1.
\end{description}
\end{teorema}
\begin{proof}
Assume that $r\leq T$. Using invariance properties of Brownian motion we get, 
\begin{align}\notag
P\left(\bigcup_{0\leq t \leq r}\{W_t\geq a-bt\}| W_T=x\right)=&P\left(\bigcup_{0\leq t \leq r}\{tW_{\frac{1}{t}}\geq a-bt\}| TW_{\frac{1}{T}}=x\right)\\ \label{uno}
=&P\left(\bigcup_{\frac{1}{r}\leq t}\{W_t\geq at-b\}| W_{\frac{1}{T}}=\frac{x}{T}\right).
\end{align}
Since $r\leq T$ it is satisfied  that $\frac{1}{r}=\frac{1}{T}+v$ for some $v\geq 0$, therefore we can rewrite the last expression in \eqref{uno} as 
\begin{align} \notag
P&\left(\bigcup_{\frac{1}{r}\leq t}\{W_t\geq at-b\}| W_{\frac{1}{T}}=\frac{x}{T}\right)\\\notag
=&P\left(\bigcup_{v\leq t-\frac{1}{T}}\left\{W_{\left(t-\frac{1}{T}\right)+\frac{1}{T}}-W_{\frac{1}{T}}\geq a\left(t-\frac{1}{T}\right)+\frac{a}{T}-b-\frac{x}{T}\right\}| W_{\frac{1}{T}}=\frac{x}{T}\right)\\\notag
=&P\left(\bigcup_{v\leq t}\left\{W_{t+\frac{1}{T}}-W_{\frac{1}{T}}\geq at+\frac{a}{T}-b-\frac{x}{T}\right\}| W_{\frac{1}{T}}=\frac{x}{T}\right). 
\end{align}
Using the independence of increments of brownian motion we conclude 
\begin{equation}
\label{res_p}
P\left(\bigcup_{0\leq t \leq r}\{W_t\geq a-bt\}| W_T=x\right)=P\left(\bigcup_{v\leq t}\left\{W_{t}\geq at+\frac{a}{T}-b-\frac{x}{T}\right\}\right).
\end{equation}
Observe that $r=T$ implies that $v=0$, therefore we can apply case \textbf{i)} of Proposition \ref{clasico} to the right-hand side of equation \eqref{res_p} and thus \textbf{a)} follows immediately. 

Now, in equality \eqref{res_p}, let as assume that $v>0$, i.e., $r<T$ and, to simplify notation, we set $\tilde{b}=\frac{a}{T}-b-\frac{x}{T}$. Conditioning on the value of random variable $W_v$  and using once again invariance properties and independence of increments of brownian motion, we obtain
\begin{align}\notag
P&\left(\bigcup_{v\leq t}\left\{W_{t}\geq at+\tilde{b}\right\}\right) \\ \notag
&\qquad =\int_{-\infty}^{\infty}P\left(\bigcup_{v\leq t}\left\{W_{t}\geq at+\tilde{b}\right\}|W_{v}=y\right)P(W_v\in dy)\\ \notag
&\qquad =\int_{-\infty}^{\infty}P\left(\bigcup_{0\leq t}\left\{W_{t+v}-W_v\geq at+av+\tilde{b}-y\right\}\right)P(W_v\in dy)\\ \label{dos}
&\qquad =\int_{-\infty}^{\infty}P\left(\bigcup_{0\leq t}\left\{W_{t}\geq at+av+\tilde{b}-y\right\}\right)P(W_v\in dy).
\end{align}
Assuming that $a>0$ and applying the case \textbf{i)} of Proposition \ref{clasico} in \eqref{dos} we get
\begin{align}\notag
P&\left(\bigcup_{v\leq t}\left\{W_{t}\geq at+\tilde{b}\right\}\right) \\ \label{tres}
&\qquad =\int_{-\infty}^{av+\tilde{b}}\exp\left(-2\left\{av+\tilde{b}-y\right\}a\right)\phi(0,v)dy+\int^{\infty}_{av+\tilde{b}}\phi(0,v)dy.
\end{align}
Using the properties of the normal distribution and the following integration formula
\begin{equation*}
\label{if}
\int_{-\infty}^k\exp(-\{ax^2+bx\})dx=\exp\left(\frac{b^2}{4a}\right)\sqrt{\frac{\pi}{a}}\Phi\left(\frac{2ka+b}{\sqrt{2a}}\right)
\end{equation*}
 in expression \eqref{tres}, we get 
\begin{equation*}
\begin{split}
P\left(\bigcup_{v\leq t}\left\{W_{t}\geq at+\tilde{b}\right\}\right)=&\exp\left(-2a\tilde{b}\right)\Phi\left(\frac{\tilde{b}}{\sqrt{v}}-a\sqrt{v}\right)\\
&+1-\Phi\left(\frac{\tilde{b}}{\sqrt{v}}+a\sqrt{v}\right).
\end{split}
\end{equation*}
Finally the result is obtained if we replace $\tilde{b}=\frac{a}{T}-b-\frac{x}{T}$ in the above equality. This ends the proof.
\end{proof}

\begin{obs}
If $r\to T^{-}$ in \eqref{bbcp2c}, we readily obtain the result \eqref{bb}.
\end{obs}

The second result is the following. 

\begin{teorema}
\label{teotr}
Let be $T<r$. If $a>0$ and $x<a-bT$ then
\begin{equation}
\label{Tmenor}
\begin{split}
P&\left(\bigcup_{0\leq t \leq r}\{W_t\geq a-bt\}| W_T=x\right)\\
&\quad=\exp\left(-2a(a-x-bT)/T\right)\left[\Phi\left(\frac{a-x-bT}{\sqrt{r-T}}-b\sqrt{r-T}\right)\right.\\
&\qquad \left.-\exp\left(2(a-x-bT)b\right)\Phi\left(-b\sqrt{r-T}-\frac{a-x-bT}{\sqrt{r-T}}\right)\right]+1-\Phi\left(\frac{a-x-bT}{
\sqrt{r-T}}-b\sqrt{r-T}\right)\\
&\qquad+\exp\left(2(a-x-bT)b\right)\Phi\left(-b\sqrt{r-T}-\frac{a-x-bT}{\sqrt{r-T}}\right).
\end{split}
\end{equation}
If $x\geq a-bT$ then the probability is $1$.
\end{teorema}
\begin{proof}
First, we observe that 
\begin{equation}
\label{unio}
\begin{split}
P&\left(\bigcup_{0\leq t \leq r}\{W_t\geq a-bt\}| W_T=x\right)\\
&\qquad= P\left(\bigcup_{0\leq t\leq T}\left\{W_t\geq a-bt\right\}|W_T=x\right)+P\left(\bigcup_{T \leq t\leq r}\left\{W_t\geq a-bt\right\}|W_T=x\right)\\
&\qquad \quad-P\left(\bigcup_{0\leq t\leq T}\left\{W_t\geq a-bt\right\}, \bigcup_{T \leq t\leq r}\left\{W_t\geq a-bt\right\} | W_T=x\right).
\end{split}
\end{equation}
Now, we use the Markov property and invariance properties of Brownian motion to calculate the joint probability in the last term of \eqref{unio}.
\begin{equation}
\label{MP}
\begin{split}
&\qquad P\left(\bigcup_{0\leq t\leq T}\left\{W_t\geq a-bt\right\}, \bigcup_{T \leq t\leq r}\left\{W_t\geq a-bt\right\} | W_T=x\right)\\
&\qquad =P\left(\bigcup_{0\leq t\leq T}\left\{W_t\geq a-bt\right\}| W_T=x\right)P\left(\bigcup_{T \leq t\leq r}\left\{W_t\geq a-bt\right\} | W_T=x\right)\\
&\qquad =P\left(\bigcup_{0\leq t\leq T}\left\{W_t\geq a-bt\right\}| W_T=x\right)P\left(\bigcup_{0\leq t\leq r-T}\left\{W_t\geq a-bT-x-bt\right\}\right)\\
&\qquad =P\left(\bigcup_{0\leq t\leq T}\left\{W_t\geq a-bt\right\}| W_T=x\right)\left[1-\Phi\left(\frac{a-x-bT}{\sqrt{r-T}}-b\sqrt{r-T}\right)\right.\\
&\qquad\qquad\qquad \qquad\qquad\left.+\exp\left(2(a-x-bT)b\right)\Phi\left(-b\sqrt{r-T}-\frac{a-x-bT}{\sqrt{r-T}}\right)\right],
\end{split}
\end{equation}
where the last equality in \eqref{MP} is consequence of case \textbf{ii)} of Proposition \ref{clasico}. Therefore, from \eqref{unio} and \eqref{MP} we get
\begin{equation*}
\begin{split}
P&\left(\bigcup_{0\leq t \leq r}\{W_t\geq a-bt\}| W_T=x\right)\\
&\qquad =P\left(\bigcup_{0\leq t\leq T}\left\{W_t\geq a-bt\right\}|W_T=x\right)\left[\Phi\left(\frac{a-x-bT}{\sqrt{r-T}}-b\sqrt{r-T}\right)\right.\\
&\qquad\qquad\qquad \qquad\qquad\left.-\exp\left(2(a-x-bT)b\right)\Phi\left(-b\sqrt{r-T}-\frac{a-x-bT}{\sqrt{r-T}}\right)\right]\\
&\qquad \quad+1-\Phi\left(\frac{a-x-bT}{\sqrt{r-T}}-b\sqrt{r-T}\right)+\exp\left(2(a-x-bT)b\right)\Phi\left(-b\sqrt{r-T}-\frac{a-x-bT}{\sqrt{r-T}}\right).
\end{split}
\end{equation*}
Finally, we use equality \eqref{bb} to obtain
\begin{equation*}
\begin{split}
P&\left(\bigcup_{0\leq t \leq r}\{W_t\geq a-bt\}| W_T=x\right)\\
&\qquad=\exp\left(-2a(a-x-bT)/T\right)\left[\Phi\left(\frac{a-x-bT}{\sqrt{r-T}}-b\sqrt{r-T}\right)\right.\\
&\qquad\qquad\qquad \qquad\qquad\left.-\exp\left(2(a-x-bT)b\right)\Phi\left(-b\sqrt{r-T}-\frac{a-x-bT}{\sqrt{r-T}}\right)\right]\\
&\qquad \quad+1-\Phi\left(\frac{a-x-bT}{\sqrt{r-T}}-b\sqrt{r-T}\right)+\exp\left(2(a-x-bT)b\right)\Phi\left(-b\sqrt{r-T}-\frac{a-x-bT}{\sqrt{r-T}}\right).
\end{split}
\end{equation*}

For the case $x\geq a-bT$ it is satisfied that
\begin{equation*}
\begin{split}
P&\left(\bigcup_{0\leq t\leq T}\left\{W_t\geq a-bt\right\}, \bigcup_{T \leq t\leq r}\left\{W_t\geq a-bt\right\} | W_T=x\right)\\&\qquad\qquad=P\left(\bigcup_{0\leq t\leq T}\left\{W_t\geq a-bt\right\}| W_T=x\right)
\end{split}
\end{equation*}
and in consequence we get that the probability \eqref{unio} is equal to $1$. This finishes the proof. 
%.
\end{proof}

\begin{obs}
If $r\to T^+$ in expression \eqref{Tmenor}, we readily obtain the result \eqref{bb}.
\end{obs}

\begin{obs}
Theorem \ref{teotr} is equivalent to a particular case of the result presented  in \cite[Theorem 2.4]{abundo}.
\end{obs}

Finally, we are ready to prove the main result of the paper. To simplify the notation we define for $x\in\mathbb{R}$
\begin{equation*}
R(t,x)=\frac{1}{c_2(p-1)\left(g(x)\right)^{p-1}}-\frac{c_1}{c_2}t.
\end{equation*}

\begin{proof}[Proof of the Theorem \ref{FDdetau}]
Let $r\geq 0$. Then
\begin{align}\notag
P(\tau\leq r)&=\int_{-\infty}^{\infty}P\left(\inf_{0\leq t}\left\{W_t\geq R(t,W_T)\right\}\leq r | W_{T}=x\right)P\left(W_T \in dx\right)\\\notag
&=\int_{-\infty}^{\infty}P\left(\inf_{0\leq t}\left\{W_t\geq R(t,x)\right\}\leq r | W_{T}=x\right)P\left(W_T \in dx\right)\\\label{Dis1}
&=\int_{-\infty}^{\infty}P\left(\bigcup_{t\in[0,r]}\left\{W_t\geq R(t,x)\right\}| W_{T}=x\right)P\left(W_T \in dx\right).
\end{align}

The second equality in \eqref{Dis1} is consequence of a well-known result, see \cite[Property CE10, pp. 462]{pfeiffer}. 

First, we suppose that $r<T$. From case \textbf{b)} of Theorem \ref{S/A} we have
\begin{equation}
\label{casobT32}
\begin{split}
P&\left(\bigcup_{t\in[0,r]}\left\{W_t\geq R(t,x)\right\}| W_{T}=x\right)\\
&=1-\Phi\left(\frac{R(T,x)-x}{T\sqrt{v}}+\frac{\sqrt{v}}{c_2(p-1)\left(g(x)\right)^{p-1}}\right)\\
&\quad +\exp\left(\frac{-2\left(R(T,x)-x\right)}{Tc_2(p-1)\left(g(x)\right)^{p-1}}\right)\Phi\left(\frac{R(T,x)-x}{T\sqrt{v}}-\frac{\sqrt{v}}{c_2(p-1)\left(g(x)\right)^{p-1}}\right).
\end{split}
\end{equation}
In other hand, the random variable $W_T\sim \phi(0,T)$ and note that $a(x)=R(T,x)-x$. Therefore replacing equation \eqref{casobT32} in the last equality in \eqref{Dis1} we get the case \textbf{i)}, i.e.,
\begin{align}\notag
P&(\tau\leq r)\\\notag
&=1-\int_{-\infty}^{\infty}\Phi\left(\frac{a(x)}{T\sqrt{v}}+\frac{\sqrt{v}}{c_2(p-1)\left(g(x)\right)^{p-1}}\right)\phi(0,T)dx\\\notag
&+\int_{-\infty}^{\infty}\exp\left(\frac{-2a(x)}{Tc_2(p-1)\left(g(x)\right)^{p-1}}\right)\Phi\left(\frac{a(x)}{T\sqrt{v}}-\frac{\sqrt{v}}{c_2(p-1)\left(g(x)\right)^{p-1}}\right)\phi(0,T)dx.
\end{align}

Now, we assume that $T<r$. From Theorem \ref{teotr}, if $x<R(T,x)$, we obtain
\begin{equation}
\label{casobTotro}
\begin{split}
P&\left(\bigcup_{t\in[0,r]}\left\{W_t\geq R(t,x)\right\}| W_{T}=x\right)\\
&\quad=\exp\left(\frac{-2\left(R(T,x)-x\right)}{Tc_2(p-1)\left(g(x)\right)^{p-1}}\right)\left[\Phi\left(\frac{R(T,x)-x}{\sqrt{r-T}}-\frac{c_1\sqrt{r-T}}{c_2}\right)\right.\\
&\qquad \qquad\left.-\exp\left(\frac{2\left(R(T,x)-x\right)c_1}{c_2}\right)\Phi\left(\frac{-R(T,x)+x}{\sqrt{r-T}}-\frac{c_1\sqrt{r-T}}{c_2}\right)\right]\\
&\qquad +1-\Phi\left(\frac{R(T,x)-x}{\sqrt{r-T}}-\frac{c_1\sqrt{r-T}}{c_2}\right)\\
&\qquad+\exp\left(\frac{2\left(R(T,x)-x\right)c_1}{c_2}\right)\Phi\left(\frac{-R(T,x)+x}{\sqrt{r-T}}-\frac{c_1\sqrt{r-T}}{c_2}\right)
\end{split}
\end{equation}
and
\begin{equation}
\label{Totro2}
P\left(\bigcup_{t\in[0,r]}\left\{W_t\geq R(t,x)\right\}| W_{T}=x\right)=1
\end{equation}
if $x\geq R(T,x)$. Therefore we separate the integral in expression \eqref{Dis1}
\begin{equation*}
\label{Dis4}
\begin{split}
P(\tau\leq r)&=\\
&=\int_{-\infty}^{\infty}P\left(\bigcup_{t\in[0,r]}\left\{W_t\geq R(t,x)\right\}| W_{T}=x\right)I_{\{x<R(T,x)\}}\phi(0,T)dx\\
&\quad+\int_{-\infty}^{\infty}I_{\{x\geq R(T,x)\}}\phi(0,T)dx,
\end{split}
\end{equation*}
and replacing expressions \eqref{casobTotro} and \eqref{Totro2} in the above equality to get
\begin{equation*}
\label{Dis3}
\begin{split}
&P(\tau\leq r)\\
&=\int_{-\infty}^{\infty}\exp\left(\frac{-2a(x)}{Tc_2(p-1)\left(g(x)\right)^{p-1}}\right)\Phi\left(\frac{a(x)}{\sqrt{r-T}}-\frac{c_1\sqrt{r-T}}{c_2}\right)I_{\{x<R(T,x)\}}\phi(0,T)dx\\
&\quad-\int_{-\infty}^{\infty}\exp\left(\frac{-2a(x)}{Tc_2(p-1)\left(g(x)\right)^{p-1}}+\frac{2a(x)c_1}{c_2}\right)\\
&\hspace{4.8cm}\times\Phi\left(\frac{-a(x)}{\sqrt{r-T}}-\frac{c_1\sqrt{r-T}}{c_2}\right)I_{\{x<R(T,x)\}}\phi(0,T)dx\\
&\quad+1-\int_{-\infty}^{\infty}\Phi\left(\frac{a(x)}{\sqrt{r-T}}-\frac{c_1\sqrt{r-T}}{c_2}\right)I_{\{x<R(T,x)\}}\phi(0,T)dx\\
&\quad+\int_{-\infty}^{\infty}\exp\left(\frac{2a(x)c_1}{c_2}\right)\Phi\left(\frac{-a(x)}{\sqrt{r-T}}-\frac{c_1\sqrt{r-T}}{c_2}\right)I_{\{x<R(T,x)\}}\phi(0,T)dx,
\end{split}
\end{equation*}
where the case \textbf{iii)} is obtained using that $a(x)=R(T,x)-x$. Finally, for the case \textbf{ii)}, we compute the limit when $r\to T$. Dominated Convergence Theorem and the continuity of the standard normal distribution imply 
\begin{equation}
\begin{split}
\lim&_{r\to T^{-}}\left\{1-\int_{-\infty}^{\infty}\Phi\left(\frac{a(x)}{T\sqrt{v}}+\frac{\sqrt{v}}{c_2(p-1)\left(g(x)\right)^{p-1}}\right)\phi(0,T)dx\right.\\\notag
&+\left.\int_{-\infty}^{\infty}\exp\left(\frac{-2a(x)}{Tc_2(p-1)\left(g(x)\right)^{p-1}}\right)\Phi\left(\frac{a(x)}{T\sqrt{v}}-\frac{\sqrt{v}}{c_2(p-1)\left(g(x)\right)^{p-1}}\right)\phi(0,T)dx\right\}\\
=&1-\int_{-\infty}^{\infty}\Phi\left(\infty\right)I_{\{a(x)>0\}}\phi(0,T)dx\\
&+\int_{-\infty}^{\infty}\exp\left(\frac{-2a(x)}{Tc_2(p-1)\left(g(x)\right)^{p-1}}\right)\Phi\left(\infty\right)I_{\{a(x)>0\}}\phi(0,T)dx\\
=&\lim_{r\to T^{+}}\left\{\int_{-\infty}^{\infty}\exp\left(\frac{-2a(x)}{Tc_2(p-1)\left(g(x)\right)^{p-1}}\right)\Phi\left(\frac{a(x)}{\sqrt{r-T}}-\frac{c_1\sqrt{r-T}}{c_2}\right)I_{\{a(x)>0\}}\phi(0,T)dx\right.\\
&\quad-\int_{-\infty}^{\infty}\exp\left(\frac{-2a(x)}{Tc_2(p-1)\left(g(x)\right)^{p-1}}+\frac{2a(x)c_1}{c_2}\right)\\
&\hspace{4.7cm}\times\Phi\left(\frac{-a(x)}{\sqrt{r-T}}-\frac{c_1\sqrt{r-T}}{c_2}\right)I_{\{a(x)>0\}}\phi(0,T)dx\\
&\quad+1-\int_{-\infty}^{\infty}\Phi\left(\frac{a(x)}{\sqrt{r-T}}-\frac{c_1\sqrt{r-T}}{c_2}\right)I_{\{a(x)>0\}}\phi(0,T)dx\\
&\quad\left.+\int_{-\infty}^{\infty}\exp\left(\frac{2a(x)c_1}{c_2}\right)\Phi\left(\frac{-a(x)}{\sqrt{r-T}}-\frac{c_1\sqrt{r-T}}{c_2}\right)I_{\{a(x)>0\}}\phi(0,T)dx\right\}.
\end{split}
\end{equation}
The above concludes the proof.
\end{proof}

\begin{obs}
Using the case \textbf{a)} of Theorem \ref{S/A} in equation \eqref{Dis1} we also get the probability \textbf{ii)} in Theorem \ref{FDdetau}.
\end{obs}

%%%%%%%%%%%%%%%%%%%%%%%%%%%%%%%%%%%%%%%%%%%%%%%%%%%%%%%%%%
%%%%%%%%%%%%%%% Notas y conclusiones %%%%%%%%%%%%%%%%%%%%%%%%%%%%%%%%%%
%%%%%%%%%%%%%%%%%%%%%%%%%%%%%%%%%%%%%%%%%%%%%%%%%%%%%%%%%%
\section{Further remarks and open problems}
We devote this section to present some conclusion remarks and further discussion. 
\begin{itemize}

\item In this paper, we have considered a stochastic model for fatigue in which the initial condition is given by a non-adapted random variable to the original filtration. Although the tools of anticipating calculus have been extensively used in finance applications (see, e.g., \cite{jafari}, \cite{kohatsu1}, \cite{kohatsu}, \cite{ReylaLN}), to the best of the author knowledge, it has been rarely studied outside this context. Nevertheless, works as \cite{mannella} and references within show that studying the anticipating case in certain experiments of physics may yield results that approximate better the actual phenomena.  Here, we have presented only a first attempt in incorporating anticipating tools to the study of fatigue models. In practice, this would represent that the dominant crack is being externally affected at a certain time $T > 0$ and the knowledge of this event is known and taken into account from the very beginning at $t=0$. 

\item In this paper the distribution function of $\tau$, the time of the ultimate damage,  is provided. Accordingly, this work can be considered as an interesting application of the explosion phenomenon in anticipating SDE which is currently an area of very fruitful research. In fact, to the best of our knowledge, there exists only few works where the law of the blow up time of non anticipating SDE is studied (the interested reader can consult, for instance \cite{feller}, \cite{Kara}, \cite{Llv}). In the case of anticipating SDE, the results are far from being completed and this work represent a further step in this direction.

\item For models with explosion is relevant to have numerical schemes that reproduce this behavior. For the case of non-anticipating SDE with coefficients without dependence of the time (where the model  \eqref{PE} is included) in \cite{davila} the authors use the Euler-Maruyama scheme to simulate solutions of SDE with explosion. Indeed, they select the step time of the method base on the Osgood test for explosion (\cite{LL}, \cite{os}), i.e., 
\begin{equation*}
T_k=\frac{h}{b(X_k)}\quad\text{for}\quad h>0,
\end{equation*}
in equation
\begin{equation*}
X_{k+1}=X_k+T_kb(X_k)+\sigma(X_k)\Delta W_k,
\end{equation*}
where $b,\sigma$ are the drift and diffusion coefficients respectively. However, as far as we know, the techniques and tools used to work with no-adapted stochastic integrals cannot be carried out easily to the Euler schemes, therefore, there are only few works related on  (see for example \cite{arturoeuler},  \cite{soledad}), in consequence, this seems as an interesting and challenging problem from the mathematical point of view.

\end{itemize}
%%%%%%%%%%%%%%%%%%%%%%%%%%%%%%%%%%%%%%%%%%%%%%%%%%%%%%%%%%
%%%%%%%%%%%%%%% BIBLIOGRAPHY %%%%%%%%%%%%%%%%%%%%%%%%%%%%%%%%%%
%%%%%%%%%%%%%%%%%%%%%%%%%%%%%%%%%%%%%%%%%%%%%%%%%%%%%%%%%%
%\bibliographystyle{siam}
%\bibliography{references}

\end{document}